\documentclass{amsart}
\usepackage{graphicx} % Required for inserting images
\usepackage[usenames,dvipsnames]{pstricks}
\usepackage{pstricks-add}
\usepackage{epsfig}
\usepackage{pst-grad} % For gradients
\usepackage{pst-plot} % For axes
\usepackage[space]{grffile} % For spaces in paths
\usepackage{etoolbox} % For spaces in paths
\makeatletter % For spaces in paths
\patchcmd\Gread@eps{\@inputcheck#1 }{\@inputcheck"#1"\relax}{}{}
\makeatother
\title{Multiprojective Geometry of Compatible Triples of Fundamental and Essential Matrices}
\author{Timothy Duff, Viktor Korotynskiy, Anton Leykin, Tomas Pajdla}
\usepackage{amsmath,amsfonts,amssymb,amsthm,booktabs,hyperref,tikz,cleveref}
\usetikzlibrary{calc}
\usepackage{enumitem}

\DeclareMathOperator{\GL}{GL}
\DeclareMathOperator{\SO}{SO}
\DeclareMathOperator{\PGL}{PGL}
\DeclareMathOperator{\Ext}{Ext}

\DeclareMathOperator{\im}{im}
\DeclareMathOperator{\adj}{adj}
\DeclareMathOperator{\rank}{rank}
\DeclareMathOperator{\codim}{codim}

\DeclareMathOperator{\trace}{tr}
\DeclareMathOperator{\conv}{conv}
\DeclareMathOperator{\mdeg}{mdeg}
\DeclareMathOperator{\tw}{tw}

\newcommand{\CC}{\mathbb{C}}

\newcommand{\PP}{\mathbb{P}}
\newcommand{\ZZ}{\mathbb{Z}}
\newcommand{\calB}{\mathcal{B}}
\newcommand{\calK}{\mathcal{K}}
\newcommand{\calF}{\mathcal{F}}
\newcommand{\calV}{\mathcal{V}}
\newcommand{\calE}{\mathcal{E}}
\newcommand{\calI}{\mathcal{I}}
\newcommand{\calM}{\mathcal{M}}
\newcommand{\calS}{\mathcal{S}}
\newcommand{\simgrp}{\calS_3}
\newcommand{\ee}{\mathbf{e}}

\newcommand{\Kdel}{\Delta}

\newcommand{\regularMap}[5]{\begin{aligned} #1\colon #2 &\rightarrow #3 \\ #4 &\mapsto #5 \end{aligned}}

\theoremstyle{definition}
\newtheorem{definition}{Definition}
\newtheorem{proposition}{Proposition}
\newtheorem{example}{Example}
\newtheorem{theorem}{Theorem}
\newtheorem{remark}{Remark}
\newtheorem{problem}{Problem}

\newcommand{\matrixr}[1]{\begin{bmatrix} #1 \end{bmatrix}}

\begin{document}

\begin{abstract}
We characterize the variety of compatible fundamental matrix triples by computing its multidegree and multihomogeneous vanishing ideal.
This answers the first interesting case of a question recently posed by Br\aa telund and Rydell. 
Our result improves upon previously discovered sets of algebraic constraints in the geometric computer vision literature, which are all incomplete (as they do \emph{not} generate the vanishing ideal) and sometimes make restrictive assumptions about how a matrix triple should be scaled.
Our discussion touches more broadly on generalized compatibility varieties, whose multihomogeneous vanishing ideals are much less well understood.
One of our key new discoveries is a simple set of quartic constraints vanishing on compatible fundamental matrix triples.
These quartics are also significant in the setting of essential matrices: together with some previously known constraints, we show that they locally cut out the variety of compatible essential matrix triples.
\end{abstract} 
\maketitle

\section{Introduction}
Multiview geometry~\cite{DBLP:books/cu/HZ2004, kileel2022snapshot} is a well-studied and important framework for investigating problems of reconstructing a 3D world from 2D images~\cite{schoenberger2016sfm}.  
A central object in this theory is the {\em fundamental matrix} associated to a pair of pinhole cameras. This is a $3\times 3$ matrix that is used to test whether two points in two different images are projections of a single 3D point. 
The \emph{essential matrix} is an important special case, arising for a pair of calibrated cameras.
For $n\ge 3$ cameras, the $\binom{n}{2}$ fundamental matrices will be dependent.
Here, we study these dependencies, extending previously known results and providing complete answers when $n=3.$

Let $\calB_3 \subset \GL_3 (\CC)$ denote the Borel group of $3\times 3$ upper-triangular matrices, and $\PP (\calB_3) \subset \PGL_3 (\CC)$ its projectivization.
To any subvariety $\calK \subset \PP (\calB_3)^{\times 3}$, whose points are triples of elements of $\PP (\calB_3),$ we associate a rational map defined by component-wise matrix products,
\begin{align}
\Psi_\calK : \calK \times \SO_3 (\CC)^{\times 3} \times \left(\CC^{3}\right)^{\times 3} &\dashrightarrow 
\PP (\CC^{3\times 3})^{\times 3}
\label{eq:parametrize-compatible}\\
(K_1, K_2, K_3, R_1, R_2, R_3, c_1, c_2, c_3) &\mapsto 
\left( 
K_i^{-T} R_i [c_j-c_i]_{\times } R_j^T K_j^{-1} \mid 1\le i < j \le 3 
\right). \nonumber 
\end{align}
Above, for $c = (c_x, c_y, c_z) \in \CC^3$, we denote by $[c]_{\times }$ the skew-symmetric matrix
\begin{equation}
    \label{eq:hat-matrix}
    [c]_{\times} = \begin{pmatrix}
        0 & - c_z & c_y \\
        c_z & 0 & -c_x \\
        -c_y & c_x & 0 
    \end{pmatrix}.
\end{equation}

The entities $K_i$, $R_i$, and $c_i$ represent the \emph{calibration matrix}, \emph{rotation}, and \emph{camera center} for the $i$-th camera in an arrangement of three pinhole cameras.

\begin{definition}\label{def:calibrated-triple-variety}
With notation as above, the closed subvariety $Y_\calK = \overline{\im \Psi_\calK} \subset \left(\PP^8 \right)^{\times 3}$ is the \emph{compatible triple variety} for fundamental matrices with $\calK$-priors. 
\end{definition}
Here are three interesting instances of the compatible triple variety:
\begin{enumerate}
    \item For $\calF = \PP (\calB_3) ^{\times 3}$, the variety of \emph{compatible fundamental matrix triples} $Y_\calF.$
    \item For $\calE = \{ (I, I, I) \}$, the variety of \emph{compatible essential matrix triples} $Y_{\calE}$.
    \item Between the previous two extremes, 
    for $\Kdel = \{ (K,K,K) \mid K \in \PP (\calB_3) \}$, the variety of \emph{compatible equi-uncalibrated fundamental matrix triples} $Y_{\Kdel}.$
\end{enumerate}

These three cases, $Y_\calE \subset Y_\Kdel \subset Y_\calF$, do not exhaust the possibilities for $Y_\calK$---see e.g.~\cite{DBLP:conf/cvpr/CinDMP24} for a richer taxonomy of possible calibration priors. 

Before recalling how such varieties arise from 3D reconstruction problems in computer vision, let us state the paper's main results.
Our first result provides a fairly complete picture of the variety $Y_\calF .$

\begin{theorem}\label{thm:uncalibrated-ideal}
The variety $Y_\calF$ has dimension $18$. Its multidegree function is supported on the lattice simplex $$\conv \left\{ 7 \ee_{i} + 7\ee_{j} + 4\ee_{k} \mid \text{distinct }i,j,k\in \{ 1,2,3\}  \right\},$$ with values shown in~\Cref{fig:multidegrees} (left).
The $\ZZ^3$-homogeneous vanishing ideal $\mathcal{I} (Y_\calF)$ has an explicit set of minimal generators in total degree $3$--$7$, as tallied below:
\begin{center}
    \begin{tabular}{|c|c|c|c|c|}
    \hline 
&      (cubics)   & (quartics) & (quintics) & (septics) \\
\hline 
 $\ZZ^3$-degree        & $3\ee_i$ & $2\ee_i + \ee_j + \ee_k$ & $3\ee_i + \ee_j + \ee_k$ & $3\ee_i + 3\ee_j + \ee_k$\\
 $\#$ generators & 3 & 9 & 18 & 108 \\
 equations & \cref{eq:Fij-det} & \cref{eq:quartics} & \cref{eq:quintics} & \cref{eq:septics} \\
 \hline 
    \end{tabular}
\end{center}
\end{theorem}

\begin{figure}
   % \centering
\begin{tabular}{lcr}
\begin{tikzpicture}[scale=1.5]
\filldraw[yellow!60] (-1,0) circle (5pt );
\node at (-1,0) {9};
\filldraw[yellow!60] (0,1.224744871391589) circle (5pt);
\node at (0,1.224744871391589)  {9};
\filldraw[yellow!60] (1,0) circle (5pt);
\node at (1,0) {9};

\filldraw[yellow!60] (-1/3,0) circle (5pt);
\node at (-1/3,0) {36};
\filldraw[yellow!60] (1/3,0) circle (5pt);
\node at (1/3,0) {36};

\filldraw[yellow!60] (-1/3,2*1.224744871391589/3) circle (5pt);
\node at (-1/3,2*1.224744871391589/3) {36};
\filldraw[yellow!60] (-2/3,1.224744871391589/3) circle (5pt);
\node at (-2/3,1.224744871391589/3) {36};
\filldraw[yellow!60] (1/3,2*1.224744871391589/3) circle (5pt);
\node at (1/3,2*1.224744871391589/3) {36};
\filldraw[yellow!60] (2/3,1.224744871391589/3) circle (5pt);
\node at (2/3,1.224744871391589/3) {36};

\filldraw[yellow!60] (0,1.224744871391589/3) circle (5pt);
\node at (0,1.224744871391589/3) {81};
\end{tikzpicture}
&
\phantom{fffff}
&
\begin{tikzpicture}[scale=1.5]
\filldraw[green!60] (-1,0) circle (5pt );
\node at (-1,0) {400};
\filldraw[green!60] (0,1.224744871391589) circle (5pt);
\node at (0,1.224744871391589)  {400};
\filldraw[green!60] (1,0) circle (5pt);
\node at (1,0) {400};

\filldraw[green!60] (-1/2,0) circle (5pt);
\node at (-1/2,0) {800};
\filldraw[green!60] (1/2,0) circle (5pt);
\node at (1/2,0) {800};
\filldraw[green!60] (1/4,3*1.224744871391589/4) circle (5pt);
\node at (1/4,3*1.224744871391589/4) {800};
\filldraw[green!60] (3/4,1.224744871391589/4) circle (5pt);
\node at (3/4,1.224744871391589/4) {800};
\filldraw[green!60] (-1/4,3*1.224744871391589/4) circle (5pt);
\node at (-1/4,3*1.224744871391589/4) {800};
\filldraw[green!60] (-3/4,1.224744871391589/4) circle (5pt);
\node at (-3/4,1.224744871391589/4) {800};

\filldraw[green!60] (-1/2,1.224744871391589/2) circle (5pt);
\node at (-1/2,1.224744871391589/2) {960};
\filldraw[green!60] (0,0) circle (5pt);
\node at (0,0) {960};
\filldraw[green!60] (1/2,1.224744871391589/2) circle (5pt);
\node at (1/2,1.224744871391589/2) {960};

\filldraw[green!60] (-1/4,1.224744871391589/4) circle (5pt);
\node at (-1/4,1.224744871391589/4) {\scriptsize 2000};
\filldraw[green!60] (1/4,1.224744871391589/4) circle (5pt);
\node at (1/4,1.224744871391589/4) {\scriptsize 2000};
\filldraw[green!60] (0,1.224744871391589/2) circle (5pt);
\node at (0,1.224744871391589/2) {\scriptsize 2000};

\end{tikzpicture}
\end{tabular}
    \caption{Multidegrees of $Y_\calF$ (left) and $Y_\calE$ (right).
    Values for $Y_{\calF}$ and $\mdeg_{\calE} (5,5,1)=\mdeg_{\calE} (5,1,5)=\mdeg_{\calE} (1,5,5)=400$ are from~\Cref{thm:uncalibrated-ideal} and~\Cref{prop:essential-multidegree}, respectively.
    The remaining values for $Y_\calE $  are obtained by numerical monodromy heuristics.
    }
    \label{fig:multidegrees}
\end{figure}
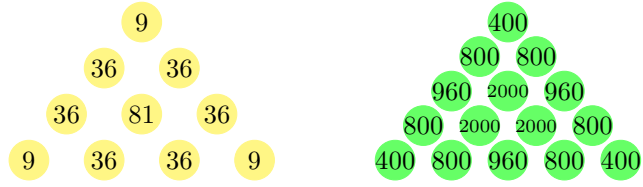

\begin{remark}\label{rem:br-conjecture}
Br\aa telund and Rydell~\cite[\S 5]{DBLP:conf/iccv/BratelundR23} recently studied a generalization of $Y_{\calF}$: given a graph $G=([n],E)$, their \emph{viewing graph variety} is the closed image of 
\begin{align}
\Psi_G : \left( \calB_3\times \SO_3 (\CC) \times \CC^{3}\right)^{\times n} &\dashrightarrow 
\PP (\CC^{3\times 3})^{\times E}
\label{eq:parametrize-viewing-graph}\\
(K_i, R_i, c_i \mid 1\le i \le n) &\mapsto 
\left( 
K_i^{-T} R_i [c_j-c_i]_{\times } R_j^T K_j^{-1} \mid ij \in E 
\right). \nonumber 
\end{align}
On a positive note, their work discovers new equations vanishing on viewing graph varieties for $n\ge 4$. On the negative side, they also prove that these equations together with the cubics and quintics in~\Cref{thm:uncalibrated-ideal} \emph{do not} cut out the viewing graph variety (even set-theoretically) for the complete graph $K_n$ on $n\ge 3$ vertices.
\Cref{thm:uncalibrated-ideal} provides a complete characterization in the first interesting case of $K_3.$
\end{remark}

Our second main result is a considerably weaker statement about the variety $Y_\calE$ of compatible essential matrix triples.
We emphasize that, relative to $Y_\calF$, describing $Y_\calE$ is objectively harder---there is no prior result to compare to. 
Recall that an irreducible subvariety $Y$ of multiprojective space is \emph{locally defined} by  multihomogeneous polynomials $f_1, \ldots f_k$ if the vanishing locus $\calV(f_1, \ldots , f_k)$ contains $Y$ and has dimension $\dim Y.$

\begin{theorem}\label{thm:calibrated-local}
The variety $Y_\calE$ has dimension $11.$ Its multidegree function is supported on the lattice simplex $$\conv \left\{ 5 \ee_{i} + 5\ee_{j} + \ee_{k} \mid \text{distinct } i,j,k\in \{ 1,2,3\} \right\},$$ as illustrated in~\Cref{fig:multidegrees} (right.)
Equations locally defining $Y_\calE$ are as follows:
\begin{center}
    \begin{tabular}{|c|c|c|c|}
    \hline 
&      (cubics)   & (quartics) & (sextics) \\
\hline 
 $\ZZ^3$-degree        & $3\ee_i$ & $2\ee_i + \ee_j + \ee_k$ & $2\ee_i + 2\ee_j + 2\ee_k$\\
 $\#$ generators & 27 & 9 & 1\\
 equations & \cref{eq:Fij-det} & \cref{eq:quartics} & \cref{eq:martyushev-sextic} \\
 \hline 
    \end{tabular}
\end{center}
\end{theorem}

\begin{remark}\label{remark:weak-characterization}
Unlike in~\Cref{thm:uncalibrated-ideal}, we do not yet fully understand the vanishing ideal $\calI (Y_\calE) .$
Our local equations for $Y_\calE$ are a ``weak characterization" in the sense of~\cite{DBLP:conf/iccv/TragerHP15}: they offer a relatively simple description of the variety that
may be sufficient for applications such as optimization over $Y_\calE$ or homotopy continuation solving~\cite{DBLP:conf/cvpr/HrubyDLP22}.
\end{remark}

\subsection{Outline}\label{subsec:outline}

\Cref{sec:background-related} recalls background and previous results so as to establish notation.
In~\Cref{sec:new-constraints}, we discuss the new quartics featuring in~\Cref{thm:uncalibrated-ideal,thm:calibrated-local}: how they were initially discovered using a combination of representation theory and numerical interpolation, their geometric interpretation, and how they fit in with previous studies.
\Cref{sec:proofs} proves the main results.
The proofs are computer-assisted, requiring (sometimes circuitous) calculations with Macaulay2~\cite{M2}.
Finally, we propose open problems for further study in~\Cref{sec:open-problems}.

Code for calculations supporting our results is hosted at the following repository: \url{https://github.com/azoviktor/compatible-fundamental-triples}.

\subsection{Previous work}\label{subsec:prev-work}

Our results build on many other previous works.
Here, we provide a short overview, with many relevant details deferred to~\Cref{sec:background-related}.

Compatible triples of fundamental matrices were first studied in the context of critical configurations for three-view projective reconstruction~\cite{DBLP:journals/ijcv/HartleyK07,DBLP:books/cu/HZ2004}.
These works identify necessary and sufficient conditions for compatibility under suitable genericity assumptions---see~\Cref{prop:hz-compatibility} and~\Cref{rem:collinear} below for precise statements.

Another fruitful line of work began in~\cite{DBLP:conf/cvpr/SenguptaAGGJSB17}, providing a rank constraint for compatibility of $n\ge 3$ cameras in terms of an associated $3n\times 3n$ matrix (when $n=3$, this is the matrix~\eqref{eq:mega-matrix} below.)
One limitation of this characterization is its assumption of \emph{proper scaling.}
In classical two-view geometry, the fundamental matrix can only be recovered up to scale---this explains why the codomain of our map~\eqref{eq:parametrize-compatible} is a product of projective spaces.
The main results of~\cite{DBLP:conf/cvpr/SenguptaAGGJSB17} and follow-up works~\cite{DBLP:conf/cvpr/KastenGGB19,DBLP:conf/iccv/KastenGGB19,DBLP:conf/cvpr/GeifmanKGB20} instead concern the affine variant of our map,
\begin{align}
\widehat{\Psi}_\calK : \calK \times \SO_3 (\CC)^{\times 3} \times \left(\CC^{3}\right)^{\times 3} &\dashrightarrow 
(\CC^{3\times 3})^{\times 3}
\label{eq:affine-map}\\
(K_1, K_2, K_3, R_1, R_2, R_3, c_1, c_2, c_3) &\mapsto 
\left( 
K_i^{-T} R_i [c_j-c_i]_{\times } R_j^T K_j^{-1} \mid 1\le i < j \le 3 
\right). \nonumber 
\end{align}
We emphasize that the closure of the image of this map, which has dimension $\dim Y_\calK$, is \emph{not} the affine cone $\widehat{Y}_\calK$ over $Y_\calK $, which has dimension $3+\dim Y_\calK .$

To circumvent scale-ambiguity, works studying $\widehat{\Psi_{\calK}}$ propose various optimization approaches that aim to recover relative scale factors for each pair of given fundamental matrices.
In contrast, our~\Cref{thm:uncalibrated-ideal} characterizes those algebraic constraints which hold for \emph{any} rescaling of a fundamental matrix triple.
This perspective is shared in recent works~\cite{DBLP:conf/iccv/BratelundR23,CONNELLY2025102446}, which identify precise compatibility conditions for any number of cameras, yet leave open the problem of set-theoretic equations.

For the setting of essential matrices in~\Cref{thm:calibrated-local}, previous works~\cite{DBLP:conf/iccv/KastenGGB19,DBLP:conf/cvpr/GeifmanKGB20,DBLP:journals/ijcv/Martyushev20} once again assume proper scaling.
Thus, it seems that our~\Cref{thm:calibrated-local} is the first result for compatible essential matrix triples that removes this restriction.

\section{Fundamentals}\label{sec:background-related}

We begin by quickly reviewing aspects of geometric computer vision, as presented in the standard text of Hartley and Zisserman~\cite{DBLP:books/cu/HZ2004}.
Later connections to algebraic geometry highlight the more recent perspective of \emph{algebraic vision} surveyed in~\cite{kileel2022snapshot}.

A pinhole camera is a surjective linear projection $\PP^3 \dashrightarrow \PP^2 $, and as such one counts $(2+1)\times (3+1) -1 = 11$ degrees of freedom.
The same count may be observed from the following ``RQ'' decomposition of a general $3\times 4$ matrix,
\begin{equation}\label{eq:matrix-decomp}
P = K R \left( I \, \vert \, -c\right),
\end{equation}
where $K\in \calB_3$, $R\in \SO_3 (\CC),$ and $c\in \CC^3.$
When $K=I$ in~\eqref{eq:matrix-decomp} is the $3\times 3$ identity matrix, we say $P$ is a \emph{calibrated} camera.
The matrix $K$ encodes the camera's \emph{intrinsic parameters} (focal length, aspect ratio, image center, and skew), while $R$ and $c$ comprise the \emph{extrinsic parameters} (orientation and location).  

As in~\eqref{eq:parametrize-compatible}, we associate a \emph{fundamental matrix} $F_{ij} \in \PP (\CC^{3\times 3})$ to any camera pair:
\begin{equation}\label{eq:Fij}
P_i \sim K_i R_i  \left( I \, \vert \, -c_i\right),
\,
P_j \sim K_j R_j  \left( I \, \vert \, -c_j\right),
\quad 
\leadsto 
\quad 
F_{ij} \sim 
K_i^{-T} R_i [c_j-c_i]_{\times } R_j^T K_j^{-1}.
\end{equation}
Here we employ the symbol $\sim $ to emphasize that $P_i, P_j,$ and $F_{ij}$ are \emph{only defined up to scale\footnote{In~\eqref{eq:Fij} and throughout, we abuse notation as follows: for any vector space $V$, we may write $u\sim v$ when each of $u$ and $v$ are either nonzero vectors in $V$ or corresponding elements of $\PP (V).$}.}
With this convention, we note the following relations:
\begin{equation}\label{eq:Fij-transpose}
    F_{ji} \sim F_{ij}^T,
\end{equation}
\begin{equation}\label{eq:Fij-det}
    \det (F_{ij}) = 0.
\end{equation}
When $P_i$ and $P_j$ are both calibrated, $F_{ij}$ is said to be an \emph{essential matrix}; throughout, we forsake the traditional notation $E_{ij}$ for the sake of uniformity.

Fundamental and essential matrices are of vital importance to the \emph{two-camera 3D reconstruction problem}: given $y_{11}, \ldots , y_{1n}, y_{21}, \ldots , y_{2n} \in \PP^2,$ can we recover cameras $P_1, P_2$ and world points $x_1, \ldots , x_n \in 
\PP^3$ with $P_i x_j \sim y_{ij}$ for all $i$ and $j?$
Observe that, if one seeks a reconstruction with uncalibrated cameras, then (exact) solutions to the reconstruction problem exist in $\PGL_4 (\CC)$-orbits,
\begin{equation}\label{eq:solution-orbit}
\PGL_4 (\CC) \cdot (P_1, P_2, x_1, \ldots , x_n) 
= 
\left\{ 
(P_1 H^{-1}, P_2 H^{-1}, H x_1, \ldots , H x_n) \mid H \in \PGL_4 (\CC)
\right\}.
\end{equation}
The fundamental matrix construction realizes a rational quotient of $\PGL_4(\CC)$ acting on the space of camera pairs.
More precisely, there exists a nonempty Zariski-open subset $\mathcal{U} \subset \PP (\CC^{3\times 4})^{\times 2}$ such that the nonempty fibers of the map 
\[
\mathcal{U} \ni (P_1, P_2) \mapsto F_{12}  \in \PP (\CC^{3\times 3})
\]
are precisely the $\PGL_4(\CC)$-orbits of points in $\mathcal{U}$. 
Thus, recovering a fundamental matrix generally means recovering a $\PGL_4(\CC)$-orbit of solutions to the reconstruction problem. This is essentially the \emph{projective reconstruction theorem}~\cite[\S 10.3]{DBLP:books/cu/HZ2004}.

Analogous remarks apply when aiming to reconstruct a calibrated camera pair: one replaces $\PGL_4 (\CC)$ with its subgroup $\simgrp$ of 3D similarity transformations, and a general essential matrix encodes a union of \emph{two} $\simgrp$-orbits which are related by the \emph{twisted pair } (see~\cite[\S 9.6]{DBLP:books/cu/HZ2004}, or~\eqref{eq:twisted-pair} below.)
In addition to the cubic constraint~\eqref{eq:Fij-det}, an essential matrix must satisfy $9$ relations given by the \emph{Demazure cubics}~\cite{demazure1988deux}, 
\begin{equation}\label{eq:demazure}
2 F_{ij} F_{ji} F_{ij} - \trace (F_{ij} F_{ji}) F_{ij} = 0.
\end{equation}

Since the two-camera reconstruction problems described above are fairly well-understood, one would naturally like a comparable understanding of what can be said for three or more cameras.
This is the primary motivation for studying the map~\eqref{eq:parametrize-compatible}.
Dimension counting (\Cref{prop:dimensions}) shows that $Y_\calF $ is not cut out by the three cubics~\eqref{eq:Fij-det}, for $1\le i < j \le 3.$ Thus, additional constraints are needed.

One well-known set of such \emph{compatibility constraints} appears in~\cite[\S 2.3]{DBLP:journals/ijcv/HartleyK07},~\cite[\S 15.4]{DBLP:books/cu/HZ2004}, and employs the language of \emph{epipolar geometry.}
These conditions give rise to the quintics in~\Cref{thm:uncalibrated-ideal}.
Consider first a general point $(F_{12}, F_{13}, F_{23}) \in Y_\calF $. 
We may assume this point lies in the image of the map $\Psi_{\calF}$, and consequently $\rank (F_{ij}) =2$ for all $i$ and $j.$
We define the $ji$-th \emph{epipole} of a fundamental matrix $F_{ij}$ to be the $3\times 1$ vector $e_{ji} \sim  \ker (F_{ij})  \in \PP^2$.
Geometrically, $e_{ji}$ may be understood as the projection of the $i$-th camera center $c_i$ under the $j$-th camera.
To see this, let $\hat{c}_i$ denote the image of $c_i$ under the usual embedding $\CC^3 \hookrightarrow \PP^3$, and calculate
\[
F_{ij} \left( P_j \hat{c}_i \right) \sim  
K_i^{-T} R_i [c_j - c_i]_\times  \left( I \, \vert \, -c_i \right) \hat{c}_i = 0.
\]
Observe that the camera centers $c_1, c_2, c_3\in \CC^3$ are \emph{non-collinear} precisely when
\begin{equation}\label{eq:noncollinearity-conditions}
e_{12} \not\sim e_{13},
\quad 
e_{21} \not\sim e_{23},
\, \,
\text{ and }
\, \, 
e_{31} \not\sim e_{32},
\end{equation}
and collinear when
\begin{equation}\label{eq:collinearity-conditions}
e_{12} \sim  e_{13},
\quad 
e_{21} \sim e_{23},
\, \,
\text{ and }
\, \, 
e_{31} \sim e_{32}.
\end{equation}
We call~\eqref{eq:noncollinearity-conditions} the \emph{noncollinearity conditions}, and~\eqref{eq:collinearity-conditions}
the \emph{collinearity conditions.}

\begin{proposition}\label{prop:hz-compatibility}\cite[Thm~15.6]{DBLP:books/cu/HZ2004}
A fundamental matrix triple $(F_{12}, F_{13}, F_{23}) $ satisfying the noncollinearity conditions~\eqref{eq:noncollinearity-conditions} lies in the image of $\Psi_\calF$ if and only if
\begin{equation}\label{eq:triangulation-conditions}
e_{13}^T F_{12} e_{23}
=
e_{12}^T F_{13} e_{32}
=
e_{21}^T F_{23} e_{31}
=
0.
\end{equation}
\end{proposition}
\begin{remark}\label{rem:collinear}
\cite[Example 3.3]{DBLP:conf/iccv/BratelundR23} shows that a fundamental matrix triple satisfying~\eqref{eq:triangulation-conditions} and the collinearity conditions~\eqref{eq:collinearity-conditions} need not lie in the image of $\Psi_\calF.$
\end{remark}

In general, two image points $y_i, y_j\in \PP^2$ with $y_i\not\sim e_{ji}, y_j\not\sim e_{ij}$ are the projections of a world point $x\in \PP^3$ when $y_i^T F_{ij} y_j=0.$ 
Thus, the \emph{triangulation conditions} of~\eqref{eq:triangulation-conditions} simply express this constraint when $x$ is one of the three camera centers.

\Cref{prop:hz-compatibility} involves open conditions on the triple $(F_{12}, F_{13}, F_{23})$: namely, the noncollinearity conditions of~\eqref{eq:noncollinearity-conditions}, as well as the implicit condition $\rank (F_{ij})=2.$
Expressing the epipoles via the adjugate matrices $\adj F_{ij}$ yields closed conditions:
\begin{equation}\label{eq:quintics}
\adj (F_{13}) F_{12} \adj (F_{32}) 
=
\adj (F_{12})  F_{13} \adj (F_{23}) 
=
\adj (F_{21})  F_{23} \adj (F_{13}) 
=
0.
\end{equation}
Equations~\eqref{eq:quintics} altogether consist of $27$ multihomogeneous polynomials---$9$ in each group with multidegrees $3\ee_{i} + \ee_{j} + \ee_{k}\in \ZZ^3.$ 
In~\Cref{thm:uncalibrated-ideal}, only two of these three groups are needed, due to relations involving the yet-to-be-described quartics.

\Cref{prop:hz-compatibility} amounts to the fact that $Y_\calF$ is locally defined by equations~\eqref{eq:Fij-det} and~\eqref{eq:quintics}.
\Cref{ex:felix-martin} below shows they do not cut out $Y_\calF $ set-theoretically.

We now move on to the septics in~\Cref{thm:uncalibrated-ideal}.
These equations are closely tied to a rank constraint first appearing in~\cite{DBLP:conf/cvpr/SenguptaAGGJSB17}, which serves as the basis of many subsequent constraints in~\cite{DBLP:conf/cvpr/KastenGGB19,DBLP:conf/iccv/KastenGGB19,DBLP:conf/cvpr/GeifmanKGB20,DBLP:journals/ijcv/Martyushev20}.
These works associate a $9\times 9$ matrix $F$ to any point in the image of the map $\widehat{\Psi_\calF}$ defined by~\eqref{eq:affine-map}.
More generally, for $(\hat{F}_{12}, \hat{F}_{13}, \hat{F}_{23}) \in \widehat{Y}_\calF$, we may associate the $9\times 9$ matrix 

\begin{equation}\label{eq:mega-matrix}
F = 
\begin{bmatrix}
    0 & \hat{F}_{12} & \hat{F}_{13} \\
    \hat{F}_{12}^T & 0 & \hat{F}_{23} \\
    \hat{F}_{13}^T & \hat{F}_{23}^T & 0 
\end{bmatrix}.
\end{equation}

Note that rescaling a single $F_{ij}$ does \emph{not} rescale $F,$ and thus conditions involving $F$ should be stated (at least \emph{a priori}) over $\widehat{Y}_\calF$ rather than $Y_\calF.$

\begin{proposition}\label{prop:rank-megamatrix}\cite{DBLP:conf/cvpr/SenguptaAGGJSB17}
For all $(\hat{F}_{12}, \hat{F}_{13}, \hat{F}_{23}) \in \widehat{Y}_\calF$, we have $\rank (F) \le 6.$
\end{proposition}

We provide a short, self-contained proof.

\begin{proof}
It will suffice, by lower-semicontinuity of rank, to prove the claim when 
\[
\hat{F}_{ij} = u_{ij} K_i^{-T} R_i [c_j-c_i]_{\times } R_j^T K_j^{-1} ,
\quad 
K_i\in \calB_3,
\quad 
u_{ij} \in \CC^* ,
\quad 
1\le i < j \le 3,
\]
and the centers $c_1,c_2,c_3\in \CC^3$ are non-collinear.
Furthermore, we may assume that $F$ is skew-symmetric, due to the identity
\begin{equation}\label{eq:mega-matrix-centers}
\begin{bmatrix}
    0 & u_{12} [c_2-c_1]_\times & u_{13} [c_3 - c_1]_\times  \\
     u_{12} [c_1-c_2]_\times &   
      0 & u_{23} [c_3-c_1]_\times \\
     u_{13} [c_1-c_3]_\times  & u_{23} [c_2-c_3]_\times &    0 
\end{bmatrix}
=
\left(D_{R,K}\right)^T F D_{R,K},
\end{equation}
where $D_{R,K}$ is an invertible matrix with diagonal blocks $R_i K_i$.
We conclude upon verifying that kernel of the matrix~\eqref{eq:mega-matrix-centers} contains three independent vectors,
 \begin{equation}\label{eq:kernel-mega-matrix}
 \begin{bmatrix}
u_{23} (c_2 - c_1) \\
u_{13} (c_1 - c_2) \\
0
 \end{bmatrix},
 \quad 
  \begin{bmatrix}
u_{23} (c_3 - c_2) \\
0\\
u_{12} (c_2 - c_3) 
 \end{bmatrix},
 \quad 
  \begin{bmatrix}
  0\\
u_{13} (c_3 - c_2) \\
u_{12} (c_2 - c_3) 
 \end{bmatrix}.
 \end{equation}

\end{proof}
\Cref{prop:rank-megamatrix} implies that the $7\times 7$ minors of $F$ all vanish on $\widehat{Y}_\calF .$
Not all of these equations vanish on $Y_\calF$, for the simple reason that they are not multihomogeneous.
Still, a large subset of these minors \emph{are} multihomogeneous: one may take
\begin{equation}\label{eq:septics}
\det \left( F_{ \left[  R^c, C^c \right]}\right) = 0,
\quad 
F \text{ as in~\eqref{eq:mega-matrix}},
\end{equation}
and $R,C \subset [9]$ indexing two deleted rows and columns having the form
\begin{equation}
\begin{split}
&R = \{ 
3(k_1-1) + i_1, \, 
3(k_2-1) + i_2
\}, \\ 
&C = \{ 
3(k_1-1) + j_1,
3(k_2-1) + j_2
\}, \\
&\text{with } 
1 \le i_1 \le j_1 \le 3,
\quad 
1 \le i_2 \le j_2 \le 3,
\quad 
1 \le k_1 < k_2 \le 3.
\end{split}\label{eq:septic-indices}
\end{equation}
The $108$ septics described in~\eqref{eq:septics},~\eqref{eq:septic-indices} are exactly those appearing in~\Cref{thm:uncalibrated-ideal}.

The matrix~\eqref{eq:mega-matrix} also appears prominently in the study of compatibility for essential matrices~\cite{DBLP:conf/iccv/KastenGGB19,DBLP:conf/cvpr/GeifmanKGB20,DBLP:journals/ijcv/Martyushev20}.
Most relevant to our work are the equations of Martyushev.
Their role is analogous to the previously-discussed septics:
despite the original assumption of proper scaling, we can use partial information from these equations to construct new equations (in this case, just one) with the needed multihomogeneity.

Following~\cite{DBLP:journals/ijcv/Martyushev20}, we define a ``diamond operator" on $n\times n$ matrices,
\begin{equation}\label{eq:diamond}
A \diamond B = \adj (A-B) - \adj (A) - \adj (B).
\end{equation}
Martyushev's equations can be written as follows: 
\begin{align}
\label{eq:martyushev-cubic}
\trace (\hat{F}_{12}\hat{F}_{23}\hat{F}_{31}) &= 0,\\
\label{eq:necF2}
\hat{F}_{ij}^T \hat{F}_{ij}\hat{F}_{jk} - \frac{1}{2}\trace(\hat{F}_{ij}^T \hat{F}_{ij})\, \hat{F}_{jk} + \adj (\hat{F}_{ij})\hat{F}_{ki}^T &= 0,\\
\label{eq:necF3}
\hat{F}_{jk}^T \adj (\hat{F}_{ij}) + \adj(\hat{F}_{jk})\hat{F}_{ij}^T + (\hat{F}_{ij}\hat{F}_{jk}) \diamond \hat{F}_{ki}^T &= 0,\\
\label{eq:necF4}
\trace^2(\hat{F}^2) - 16\trace(\hat{F}^4) + 24\sum\limits_{i < j}\trace^2(\hat{F}_{ij}^T \hat{F}_{ij}) &= 0,\\
\label{eq:martyushev-sextic}
\trace^3(F^2) - 12\trace(F^2)\trace(F^4) + 32\trace(F^6) &= 0,
\end{align}
for all distinct $i, j, k \in \{1, 2, 3\}$. 
\begin{remark}\label{remark:result-of-Martyushev}
The main result of Martyushev~\cite[Theorems 4 and 5]{DBLP:journals/ijcv/Martyushev20} implies that $(F_{12}, F_{13}, F_{23}) \in Y_\calE,$ if and only if there exists a representative
$(\hat{F}_{12}, \hat{F}_{13}, \hat{F}_{23}) \in \widehat{Y}_\calE$ with each $\hat{F}_{ij} \ne 0$ and where~\cref{eq:demazure,eq:martyushev-cubic,eq:necF2,eq:necF3,eq:necF4,eq:martyushev-sextic} all vanish.
\end{remark}

Note that~\eqref{eq:necF2}~\eqref{eq:necF3} consist of nine equations each.
Among Martyushev's equations, only~\eqref{eq:martyushev-cubic} is multihomogeneous, of degree $\ee_1 + \ee_2 + \ee_3.$
However, a multihomogeneous equation $\calM_6 (F)$ may be extracted from the sextic~\eqref{eq:martyushev-sextic} by taking its degree-$(2\ee_1 + 2\ee_2 + 2\ee_3)$ part : explicitly, 
\begin{equation}\label{eq:homogenized-martyushev}    
\calM_6 (F) :=
\pi 
\left( \trace^3(F^2) - 12\trace(F^2)\trace(F^4) + 32\trace(F^6) \right),
\end{equation}
where $\pi : \CC [F]_{6} \to \CC [F_{12}, F_{13}, F_{23}]_{(2,2,2)}$ is the natural projection.
A simple Gr\"{o}bner basis computation checks the ideal membership
\[
\calM_6 (F) - \left( \trace^3(F^2) - 12\trace(F^2)\trace(F^4) + 32\trace(F^6) \right)
\in 
\Big\langle 
\text{equations \eqref{eq:demazure}, }  1\le i < j \le 3 
\Big\rangle .
\]
Martyushev's result, \Cref{remark:result-of-Martyushev}, implies that for any $(F_{12}, F_{13}, F_{23}) \in Y_\calE,$ there exists a nonzero representative
$(\hat{F}_{12}, \hat{F}_{13}, \hat{F}_{23}) \in \widehat{Y}_\calE$ where $\mathcal{M}_6 (F)$ vanishes.

Since $\mathcal{M}_6 (F)$ is multi-homogeneous, we have $\mathcal{M}_6 (F)$ vanishing on $\widehat{Y}_\calE$ and $Y_\calE .$

\section{New Quartic Constraints}\label{sec:new-constraints}

We now introduce the quartics of~\Cref{thm:uncalibrated-ideal},

\begin{equation}\label{eq:quartics}
\begin{split}
F_{12} \adj (F_{32}) F_{31} - F_{13} \adj (F_{23}) F_{21} & = \\
F_{31} \adj (F_{21}) F_{23} - F_{32} \adj (F_{12}) F_{13}
& = \\
F_{21} \adj (F_{31}) F_{32} - F_{23} \adj (F_{13}) F_{12}
& = 0. 
\end{split}
\end{equation}

\begin{remark}\label{remark:quartics-from-symmetric-matrices}
Note that~\eqref{eq:quartics} can be formulated concisely as the following condition: 
\begin{equation}\label{equation:quartic-sandwich}
\text{for all permutations $(i,j,k)$ of $(1,2,3)$, }
F_{ij} \adj (F_{kj}) F_{ki} 
\text{ is symmetric.}    
\end{equation}
Thus three $3\times 3$ matrix equations \eqref{eq:quartics} provide a total of $9$ distinct quartic polynomial equations.
These $9$ equations are the form of the quartics we first observed, after analyzing the results of a computation described in \Cref{subsec:discovery}. 
\end{remark}

\subsection{Interpolation with symmetry}\label{subsec:discovery}

We used numerical interpolation combined with representation theory to discover the quartic constraints \eqref{eq:quartics}. Specifically, we aim to compute a basis of the degree-$4$ component of the vanishing ideal
\[ \calI(Y_\calF)_4 = \calI(Y_\calF) \cap \CC[F_{12},F_{13},F_{23}]_4. \]
The ambient space of degree-4 polynomials in the 27 variables (the entries of $F_{12}, F_{13}, F_{23}$) has dimension
\[ \dim \CC[F_{12},F_{13},F_{23}]_4 = \binom{27+4-1}{4} = 27405. \]
Naive interpolation would require computing the nullspace of a $27405 \times 27405$ Vandermonde-type matrix, which is computationally infeasible due to both memory and numerical instability. To overcome this, we leverage the symmetry of the problem. Both $Y_\calF$ and $Y_\calE$ are invariant under the action of the reductive group
\[ G = \SO_3(\CC) \times \SO_3(\CC) \times \SO_3(\CC) \times (\CC^*)^3, \]
which acts on $Y_\calF$ by
\[ \regularMap{\varphi}{G \times Y_\calF}{Y_\calF}{(S_1, S_2, S_3, \lambda_1, \lambda_2, \lambda_3, F_{12}, F_{13}, F_{23})}{\left(\lambda_1S_1F_{12}S_2^\top, \lambda_2S_1F_{13}S_3^\top, \lambda_3S_2F_{23}S_3^\top\right).} \]
We note that $\varphi$ also restricts to an action on $Y_\calE$. 
Since these actions preserve the degrees of homogeneous polynomials, $G$ acts on $\calI(Y_\calF)_4$ and $\calI(Y_\calE)_4$.

The Lie group $G$ is reductive, as its Lie algebra $\mathfrak{g}$ is reductive, and the center $Z(\mathfrak{g}) = \CC^3$ acts by diagonalizable endomorphisms on homogeneous polynomials (indeed, it acts by scalings). Thus, the $G$-module $\CC[F_{12},F_{13},F_{23}]_4$ is completely reducible \cite[Chapter 2, Section 6]{lie-algebras-humphreys}, admitting the isotypic decomposition
\begin{equation}\label{eq:isotypics} \CC[F_{12},F_{13},F_{23}]_4 = \bigoplus_{j=1}^k V_j. \end{equation}
Its $G$-invariant subspace $\calI(Y_\calF)_4$
decomposes accordingly \cite[Chapter 1, p. 70]{compact-lie-groups}:
\[ \calI(Y_\calF)_4 = \bigoplus_{j=1}^k \calI(Y_\calF)_4 \cap V_j. \]
To interpolate polynomials in $\calI(Y_\calF)_4$, it thus suffices to compute each isotypic component $\calI(Y_\calF)_4 \cap V_j$ independently. To obtain a basis of $\calI(Y_\calF)_4 \cap V_j$ we pick a basis $\{v_1, \dots, v_{m_j}\}$ of $V_j$, sample $m_j$ generic points $\{p_1,\dots,p_{m_j}\} \subset Y_\calF$, construct a Vandermonde-type matrix
\[ A = \begin{bmatrix}
    v_1(p_1) & \dots & v_{m_j}(p_1) \\ \vdots & \ddots & \vdots \\ v_1(p_{m_j}) & \dots & v_{m_j}(p_{m_j})
\end{bmatrix} \in \CC^{m_j \times m_j} \]
and compute its nullspace. Then, for every vector $(n_1, \dots, n_{m_j})$ from a basis of $\mathrm{null}(A)$ we obtain a basis element of $\calI(Y_\calF)_4 \cap V_j$ as $\sum_{i=1}^{m_j} n_iv_i$, thereby obtaining a full basis of $\calI(Y_\calF)_4 \cap V_j$.

Using techniques from representation theory, we decompose $\CC[F_{12},F_{13},F_{23}]_4$ into isotypic components and find that
\[ k = 372, \quad \max_{j \in [k]} m_j = 375. \]
In contrast to the naive $27405 \times 27405$ interpolation problem, isotypic decomposition reduces the largest nullspace computation to a manageable $375 \times 375$ system.

Although this $375 \times 375$ nullspace computation is already computationally tractable, we can reduce the size of the problem even further by leveraging the structure of highest weight vectors within each isotypic component. Let $H_j$ be the highest weight space of $V_j$, i.e., if $V_j = W_j^{\oplus a_j}$, where $W_j$ is irreducible, then $H_j$ is the direct sum of the 1-dimensional highest weight spaces of $W_j$'s. So, $\dim H_j = a_j$. It is enough to compute a basis of
\[ H_j' = \calI(Y_\calF)_4 \cap H_j \]
since all the constraints can be obtained by applying the group action:
\[ \calI(Y_\calF)_4 \cap V_j = G \cdot H_j' := \mathrm{span}_\CC \, \{g\cdot h' \;|\; g \in G, h' \in H_j'\}. \]
If we look at the maximum dimension of the highest weight space across all isotypic components, we see that
\[ \max_{j \in [k]} a_j = 5. \]
This means that, by working with highest weight vectors, the interpolation task reduces to computing nullspaces of size at most $5 \times 5$, a dramatic improvement over both the original $27405 \times 27405$ problem and the already-manageable $375 \times 375$.

For decomposing $\CC[F_{12}, F_{13}, F_{23}]_4 \cong \mathrm{Sym}^4(\CC[F_{12}, F_{13}, F_{23}]_1)$ into irreducible subrepresentations (and then into isotypic components) we proceed as follows:
\begin{enumerate}[label=(\alph*)]
    \item Since $\mathfrak{g}$ is reductive, it has a root space decomposition,
    \[ \mathfrak{g} = \mathfrak{h} \, \oplus \, \underbrace{\bigoplus_{\alpha \in \Phi^+} \mathfrak{g}_\alpha}_{\mathfrak{n}^+} \, \oplus \, \underbrace{\bigoplus_{\alpha \in \Phi^-} \mathfrak{g}_\alpha}_{\mathfrak{n}^-},  \]
    where $\mathfrak{h}$ is a 6-dimensional Cartan subalgebra of $\mathfrak{g}$, indices $\Phi^+$ (resp. $\Phi^-$) are positive (resp. negative) roots, and $\mathfrak{n}^+$ (resp. $\mathfrak{n}^-$) is positive (resp. negative) root space \cite[Chapters 7,8]{brian-hall-representation-theory}. If we represent $\mathfrak{g}$ as
    \[ \mathfrak{g} = \underbrace{\mathfrak{so}(3, \CC)}_{\mathfrak{g}_1} \, \oplus \, \underbrace{\mathfrak{so}(3, \CC)}_{\mathfrak{g}_2} \, \oplus \, \underbrace{\mathfrak{so}(3, \CC)}_{\mathfrak{g}_3} \, \oplus \, \underbrace{\CC^3}_{\mathfrak{g}_4}, \]
    then 
    \[ \mathfrak{h} = \oplus_{j=1}^4 \mathfrak{h}_j, \quad \mathfrak{n}^+ = \oplus_{j=1}^4 \mathfrak{n}^+_j, \quad \mathfrak{n}^- = \oplus_{j=1}^4 \mathfrak{n}^-_j, \] where the $\mathfrak{h}_j$ are Cartan subalgebras of $\mathfrak{g}_j$, and each $\mathfrak{n}^+_j$ and $\mathfrak{n}^-_j$ are the positive and negative root spaces of $\mathfrak{g}_j$. Since $\mathfrak{g}_4$ is abelian, we have that $\mathfrak{h}_4 = \mathfrak{g}_4$ and both $\mathfrak{n}^+_4$ and $\mathfrak{n}^-_4$ are trivial. As for $j=1,2,3$, we fix $\mathfrak{h}_j = \langle J_3 \rangle$, $\mathfrak{n}^+_j = \langle J_+ \rangle$, $\mathfrak{n}^-_j = \langle J_- \rangle$, where
    \[
    J_3 = \matrixr{0 & -i & 0 \\ i & 0 & 0 \\ 0 & 0 & 0}, \quad
    J_+ = \matrixr{0 & 0 & -1 \\ 0 & 0 & -i \\ 1 & i & 0}, \quad
    J_- = \matrixr{0 & 0 & 1 \\ 0 & 0 & -i \\ -1 & i & 0}.
    \]
    \item Decompose $V = \CC[F_{12}, F_{13}, F_{23}]_1$ into weight spaces \cite[Chapter 6, Section 20]{lie-algebras-humphreys} as
    \begin{equation} \label{eq:weight-space-decomposition}
        \CC[F_{12}, F_{13}, F_{23}]_1 = \bigoplus_{\lambda \in \mathfrak{h}^*} V_\lambda,
    \end{equation}
    where
    \[ V_\lambda = \{v \in V \;|\; X\cdot v = \lambda(X)v \;\; \forall X \in \mathfrak{h}\}. \]
    In other words, we find a basis of $\CC[F_{12}, F_{13}, F_{23}]_1$ that consists of simultaneous eigenvectors of the linear operators from $\mathfrak{h}$.
    Since we know the eigenvectors for each $\mathfrak{h}_j$, finding a simultaneous eigenbasis for $\mathfrak{h}$ is straightforward. The eigenvectors for $\mathfrak{h}_4$ are simply the variables, since $\mathfrak{g}_4$ acts by scalings.
    Each $\mathfrak{h}_j$ for $j=1,2,3$ acts on the rows or columns of two fundamental matrices while fixing the remaining one. For example, considering the column action of $\mathfrak{h}_j$ on
    \[ F = \begin{bmatrix}
        f_{11} & f_{12} & f_{13} \\
        f_{21} & f_{22} & f_{23} \\
        f_{31} & f_{32} & f_{33}
    \end{bmatrix}, \]
    the nine eigenvectors are
    \[ f_{1k} - if_{2k}, \quad f_{3k}, \quad f_{1k} + if_{2k}, \quad k = 1,2,3. \]
    For the row action, we consider the transpose of $F$ and obtain the eigenvectors similarly to the column case. For the fundamental matrix fixed under the action, the eigenvectors are simply the variables.
    \item Obtain the weight space decomposition of $\mathrm{Sym}^p(\CC[F_{12}, F_{13}, F_{23}]_1)$ (we need $p=4$) as follows:
    \begin{equation}\label{eq:sym-power-weight-decomposition}
    \begin{split}
    & \mathrm{Sym}^p(\CC[F_{12}, F_{13}, F_{23}]_1) = \bigoplus_{\mu \in \mathfrak{h}^*} W_\mu \text{ with }\\
    & W_\mu= \bigoplus_{\substack{\lambda_1 + \cdots + \lambda_p = \mu \\ \lambda_i \in \mathfrak{h}^*}} V_{\lambda_1} \odot \dots \odot V_{\lambda_p} \text{ where } \\
    & V \odot W := \mathrm{span}_\CC(\{fg \;|\; f \in V, g \in W\}).    \end{split}
    \end{equation}
    See e.g.~\cite[p. 375]{fulton-harris} for weight space decompositions of tensor products.
    \item Compute the highest weight vectors representing the irreducible subrepresentations of $\mathrm{Sym}^4(\CC[F_{12}, F_{13}, F_{23}]_1)$ as follows. For each weight space $W_\mu$ in the decomposition above, we find the subspace
    \[
    H_\mu := \{ v \in W_\mu \mid X \cdot v = 0 \text{ for all } X \in \mathfrak{n}_+ \}.
    \]
The elements of $H_\mu$ are the highest weight vectors of irreducible subrepresentations contained in $\mathrm{Sym}^4(\CC[F_{12}, F_{13}, F_{23}]_1)$. When $\mu$ is the highest weight of the irreducible representation $W_j$ appearing in the isotypic component $V_j$ (defined earlier), we have
\[ H_\mu = H_j, \]
where $H_j$ is the highest weight space of $V_j$.
Once a highest weight vector is identified, its orbit under $\mathfrak{g}$ (or just $\mathfrak{n}_-$) generates the corresponding irreducible subrepresentation \cite[Proposition  14.13]{fulton-harris}. Grouping together subspaces with the same highest weight yields the isotypic components. According to the earlier remarks, it is enough to consider $H_\mu$ for constructing Vandermonde-type matrices.
\end{enumerate}

After we compute a basis of $\calI(Y_\calF)_4 \cap V_j$ for all $j \in [k] = [372]$, we see that there are $51$ non-trivial intersections $\calI(Y_\calF)_4 \cap V_j$, out of which only $3$ cannot be obtained from the variable multiples of the constraints in $\calI(Y_\calF)_3$. That is, if we consider the vector space $Q$ generated by $\{ fg \;|\; f \in \CC[F_{12},F_{23},F_{31}]_1, g \in \calI(Y_\calF)_3\}$, 48 out of 51 isotypic components $V_j$ belong to $Q$.
These 3 isotypic components correspond to the constraints that make the matrices of polynomials in \eqref{equation:quartic-sandwich} symmetric. 

Finally, we provide some comments on the numerical computations in Julia.
During the decomposition into irreducible representations, nullspace computations also arise, specifically in steps~(b) and~(d).
The largest nullspace problem occurs in step~(d) and has size $23 \times 23$.
For the SVD-based nullspace computation used in the decomposition algorithm, we set the absolute tolerance to $\varepsilon = 10^{-5}$.

To compute nullspaces of Vandermonde-type matrices, we use a stricter tolerance $\varepsilon < 10^{-10}$.
The resulting polynomials, initially obtained with floating-point coefficients, are then converted to exact-coefficient polynomials by first scaling to obtain integer coefficients and subsequently rounding to the nearest integer.

\subsection{Geometric interpretation of the new quartic constraints.}\label{subsec:geometric-meaning}
\begin{figure}
{
\psscalebox{1.0 1.0} % Change this value to rescale the drawing.
{
\begin{pspicture}(2.4,-3.7)(12.0,4.0)
\definecolor{colour0}{rgb}{0.49803922,0.49803922,0.49803922}
%\rput(0.2812791,-5.0882325){\psgrid[gridwidth=0.028222222, subgridwidth=0.014111111, gridlabels=6.0pt, subgridcolor=colour0](0,0)(0,0)(13,10)}
\pspolygon[linecolor=black, linewidth=0.04](2.331279,0.8817676)(2.331279,-2.1182325)(5.231279,-3.1182325)(5.231279,-0.11823242)
\pspolygon[linecolor=black, linewidth=0.04](11.931279,0.8817676)(11.931279,-2.1951554)(8.931279,-3.1182325)(8.931279,-0.14387345)
\pspolygon[linecolor=black, linewidth=0.04](5.3312793,0.8817676)(8.831279,0.8817676)(8.831279,3.8817675)(5.3312793,3.8817675)
\psdots[linecolor=black, dotsize=0.2](2.831279,-3.1182325)
\psdots[linecolor=black, dotsize=0.2](11.531279,-3.1182325)
\psdots[linecolor=black, dotsize=0.2](7.131279,3.1817675)
\psdots[linecolor=black, dotsize=0.2](5.231279,-3.1182325)
\psdots[linecolor=black, dotsize=0.2](8.931279,-3.1182325)
\psdots[linecolor=black, dotsize=0.2](3.831279,-1.7182324)
\psline[linecolor=black, linewidth=0.04](5.031279,0.08176758)(7.131279,3.1817675)
\psline[linecolor=black, linewidth=0.04](7.131279,3.1817675)(7.131279,3.1817675)(9.231279,0.08176758)
\psline[linecolor=black, linewidth=0.04](11.531279,-3.1182325)(2.831279,-3.1182325)
\psdots[linecolor=black, dotsize=0.2](10.431279,-1.6182324)
\psdots[linecolor=black, dotsize=0.2](6.3312793,1.9817675)
\psdots[linecolor=black, dotsize=0.2](7.931279,1.9817675)
\psline[linecolor=black, linewidth=0.04](5.3312793,1.9817675)(5.3312793,1.9817675)(8.831279,1.9817675)
\psline[linecolor=black, linewidth=0.04](5.231279,-3.1182325)(5.231279,-3.1182325)(2.331279,-0.11823242)
\psline[linecolor=black, linewidth=0.04](8.931279,-3.1182325)(8.931279,-3.1182325)(11.931279,-0.11823242)
\psline[linecolor=black, linewidth=0.04](3.831279,-1.7182324)(2.831279,-3.1182325)
\psline[linecolor=black, linewidth=0.04](11.531279,-3.1182325)(10.431279,-1.6182324)
\rput[bl](2.631279,-3.6182325){$c_1$}
\rput[bl](11.431279,-3.5182323){$c_2$}
\rput[bl](6.931279,3.4817677){$c_3$}
\rput[bl](8.831279,-3.5182323){$e_{21}$}
\rput[bl](7.931279,2.2817676){$e_{32}$}
\rput[bl](5.631279,2.2817676){$e_{31}$}
\rput[bl](4.931279,-3.5182323){$e_{12}$}
\rput[bl](4.171279,-1.7382324){$e_{13}$}
\rput[bl](10.731279,-1.7182324){$e_{23}$}
\rput[bl](2.831279,-0.4182324){$\ell_1$}
\rput[bl](11.031279,-0.5182324){$\ell_2$}
\rput[bl](6.931279,1.5817676){$\ell_3$}
\psline[linecolor=blue, linewidth=0.04](3.831279,-2.5182323)(3.831279,-0.8182324)
\psline[linecolor=blue, linewidth=0.04](3.131279,-1.5182325)(4.531279,-1.9182324)
\psline[linecolor=blue, linewidth=0.04](3.4312792,-0.91823244)(4.231279,-2.5182323)
\psline[linecolor=blue, linewidth=0.04](3.331279,-2.1182325)(4.3312793,-1.3182324)
\rput[bl](3.531279,-0.8182324){$\im (F_{13})$}
\rput[bl](6.661582,-1.8576263){$F_{12}$}
\rput[bl](4.0379457,0.78843427){$F_{13}$}
\psline[linecolor=black, linewidth=0.04](8.759494,-1.8346719)(8.821885,-1.8346719)(5.39037,-2.3243687)
\psline[linecolor=black, linewidth=0.04](9.031279,-1.79399)(10.231279,-1.6182324)
\psline[linecolor=black, linewidth=0.04, arrowsize=0.05291667cm 2.0,arrowlength=1.4,arrowinset=0.0]{->}(5.1249156,-2.3637626)(4.752491,-2.416187)
\psline[linecolor=blue, linewidth=0.04](5.091279,1.2917676)(4.391279,0.2717676)
\psline[linecolor=blue, linewidth=0.04, arrowsize=0.05291667cm 2.0,arrowlength=1.4,arrowinset=0.0]{->}(4.251279,0.071767576)(3.901279,-0.42823243)
\psline[linecolor=blue, linewidth=0.04](3.251279,-1.0782324)(4.491279,-2.3582325)
\end{pspicture}
}}
\caption{
Geometry of the quartics~\eqref{eq:quartics}: each epipolar line $l_i=F_{ij}e_{jk}$ belongs to the pencil of epipolar lines~$\im (F_{ik})$.}\label{fig:quartics-geometry}
\end{figure}
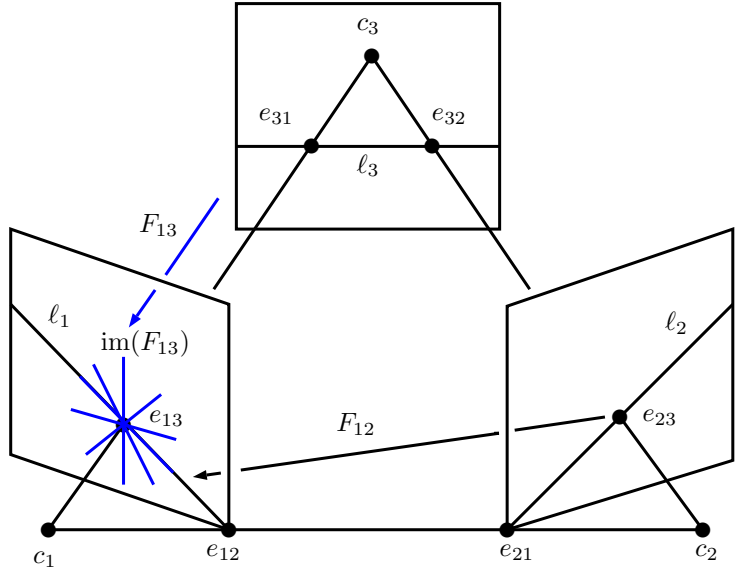

Since the fundamental matrices $F_{kj}$ have rank $2$, their adjugates $\adj(F_{kj})$ have rank one and so may be written as $\adj(F_{kj}) =  e_{jk} u^T$ for some $u\in \CC^{3\times 1}$. The matrices $F_{ij}\adj(F_{kj}) F_{ki}$ appearing in the quartics~\eqref{eq:quartics} are rank-one symmetric matrices, and thus factorize as  $F_{ij}\adj(F_{kj}) F_{ki} = \alpha \, a\,a^T$ with $a\in \CC^{3\times 1}$ and $\alpha \in \CC^* $. Thus,
\[
\alpha\, a\,  a^T = F_{ij} e_{jk} u^T F_{ki}
\quad 
\Rightarrow \quad 
\alpha\, F_{ij} e_{jk} = F_{ki}^Tu = F_{ik} u,
\]
implying that $ F_{ij} e_{jk} \in \im (F_{ik})$.
Considering all possible assignments to $i,j,k \in [3]$ without repetitions and the symmetry by~\cref{eq:quartics}, we obtain six relations  
\begin{align*}
&F_{12} e_{23} \in \im ( F_{13} ),\  F_{13}e_{32} \in \im (F_{12}),\ F_{23} e_{31} \in \im( F_{21}),\\
&F_{21} e_{13} \in \im (F_{23} ),\ F_{31} e_{12} \in \im (F_{32}),\ F_{32} e_{21} \in \im ( F_{31} ),
\end{align*}
from which the quartics~\eqref{eq:quartics} follow.

As is well known in multiview geometry, $\im (F_{ij})$ represents the \emph{pencil of epipolar lines} through the basepoint $e_{ij}$ in image $i$.
The lines in this pencil correspond to points in image $j$.
In particular, the epipolar line $F_{ij}e_{jk}$ in image $i$ corresponds to the epipole $e_{jk}\sim P_j \hat{c}_k$.
The relations above may be understood geometrically as requiring that the epipolar line $\ell_i=F_{ij}e_{jk}$ belongs to the pencil of epipolar lines $\operatorname{im} ( F_{ik})$; with this interpretation, we see that $\ell_i$ is generally the line spanned by $e_{ij}$ and $e_{ik}.$
 \Cref{fig:quartics-geometry} illustrates the specific relation $F_{12} e_{23} \in \im (F_{13} )$.

\subsection{Comparing quartics with previously-known constraints}\label{subsec:comparison}

Let us investigate to what extent the four types of constraints in~\Cref{thm:uncalibrated-ideal} are really needed.

\begin{example}\label{ex:felix-martin}
In~\cite{DBLP:conf/iccv/BratelundR23}, the following example of a triple is given,
\begin{equation}\label{eq:felix-martin}
F_{12} \sim
\begin{bmatrix}
    \phantom{-}0&1&0\\
    -1&0&0\\
    \phantom{-}0&0&0
\end{bmatrix}
, \quad 
F_{13}  \sim \begin{bmatrix}
    \phantom{-}0&0&1\\
    \phantom{-}0&0&0\\
    -1&0&0
\end{bmatrix}
, \quad 
F_{23}  \sim \begin{bmatrix}
    0&0&0\\
    0&0&0\\
    0&0&1
\end{bmatrix},
\end{equation}
for which the cubics~\eqref{eq:Fij-det} and quintics~\eqref{eq:quintics} are satisfied, yet $(F_{12}.F_{13},F_{23})\notin Y_\calF.$
This can be confirmed using~\Cref{thm:uncalibrated-ideal}, since direct evaluation of the quartics~\eqref{eq:quartics} and septics~\eqref{eq:septics} shows they do not all vanish on this triple.
\end{example}

Arguably, the triple~\eqref{eq:felix-martin} is of limited interest, since $\rank (F_{23}) = 1.$
Thus, we next provide two examples of fundamental matrix triples, with $\rank(F_{ij}) =2$ for all $i,j,$ which demonstrate the necessity of the constraints in~\Cref{thm:uncalibrated-ideal}.

\begin{example}\label{ex:septics-needed}
For the fundamental matrix triple
\begin{equation}\label{eq:septics-needed}
F_{12}  \sim
\begin{bmatrix}
   \phantom{-} 0&1&0\\
\phantom{-}4&0&3\\
    -2&0&0
\end{bmatrix}
, \quad 
F_{13}  \sim \begin{bmatrix}
    \phantom{-}0&\phantom{-}0&1\\
    -2&-3&0\\
    \phantom{-}3&-3&0
\end{bmatrix}
, \quad 
F_{23}  \sim \begin{bmatrix}
    -1&2&-1\\
    \phantom{-}0&0&\phantom{-}0\\
    \phantom{-}1&0&-1
\end{bmatrix},
\end{equation}
all cubics, quartics, and quintics in~\Cref{thm:uncalibrated-ideal} vanish, but exactly one of the 108 septics does \emph{not} vanish.
Here, neither~\eqref{eq:noncollinearity-conditions} nor~\eqref{eq:collinearity-conditions}, are satisfied, since
\begin{equation}\label{eq:bad-epipoles-1}
e_{12}  \sim \begin{bmatrix}
1\\
0\\
0
\end{bmatrix} \sim e_{13} ,
\quad 
e_{21} \sim \begin{bmatrix}
    0\\
    1 \\
    0
\end{bmatrix} \sim e_{23},
\quad 
e_{31} \sim \begin{bmatrix}
    0\\
    0\\
    1
\end{bmatrix}
\not\sim 
\begin{bmatrix}
    1\\
    1\\
    1
\end{bmatrix}
\sim 
e_{32}.
\end{equation}
\end{example}

\begin{example}\label{ex:5-epipoles}
With a similar flavor as~\Cref{ex:septics-needed}, the triple
\begin{equation}\label{eq:triple-5-epipoles}
F_{12}  \sim
\begin{bmatrix}
   \phantom{-}0&0&\phantom{-}0\\
    \phantom{-}0&0&-3\\
-3&0&\phantom{-}4
\end{bmatrix}
, \quad 
F_{13}  \sim \begin{bmatrix}
    \phantom{-}0&\phantom{-}0&0\\
    \phantom{-}2&-3&0\\
    -4&\phantom{-}4&0
\end{bmatrix}
, \quad 
F_{23}  \sim
\begin{bmatrix}
    \phantom{-}1&-1&\phantom{-}3\\
    \phantom{-}1&-1&\phantom{-}1\\
    -1&\phantom{-}1&-3
\end{bmatrix}
\end{equation}
has five distinct epipoles,
\begin{equation}\label{eq:bad-epipoles-2}
e_{12} \sim \begin{bmatrix}
1\\
0\\
0
\end{bmatrix} \sim e_{13},
\quad 
e_{21} \sim \begin{bmatrix}
    0\\
    1 \\
    0
\end{bmatrix} \not\sim  \begin{bmatrix}
    1\\
    0 \\
    1
\end{bmatrix}\sim e_{23} ,
\quad 
e_{31} \sim \begin{bmatrix}
    0\\
    0\\
    1
\end{bmatrix}
\not\sim 
\begin{bmatrix}
    1\\
    1\\
    0
\end{bmatrix}
\sim 
e_{32},
\end{equation}
and satisfies all cubics and quintics, but fails to satisfy two quartics and four septics.
\end{example}

The septics~\eqref{eq:septics}, it turns out, can be obtained by saturating the ideal generated by cubics, quartics, and quintics by an ideal of polynomials encoding the collinearity conditions~\eqref{eq:collinearity-conditions}.
In fact, this is how we initially discovered these equations, before realizing their interpretation as $7\times 7 $ minors of the matrix $F.$

\begin{example}\label{ex:quartics-needed}
Similarly to~\Cref{ex:septics-needed}, we can construct triples $(F_{12}, F_{13}, F_{23})\notin Y_{\calF}$ which satisfy all cubics, quintics, and septics, but not all quartics.
Indeed, if $F_{12}$ and $F_{13}$ are arbitrary rank-1 matrices and $F_{23}$ is an arbitrary matrix of rank at most $2,$ then a Gr\"{o}bner basis calculation shows that all cubics, quintics, and septics vanish at $(F_{12}, F_{13}, F_{23}).$
The quartics, however, generally do not vanish on such a triple: concretely, we may take
\begin{equation}\label{eq:quartics-needed}
F_{12}  \sim \begin{bmatrix}
1&0&0\\
0&0&0\\
0&0&0
\end{bmatrix},
\quad 
F_{13}  \sim \begin{bmatrix}
0&0&0\\
0&1&0\\
0&0&0
\end{bmatrix},
\quad 
F_{23}  \sim \begin{bmatrix}
4&1&-1\\
3&0&-1\\
1&1&\phantom{-}0
\end{bmatrix},
\end{equation}
showing that the quartics are needed to cut out $Y_\calE$ set-theoretically.
\end{example}

In view of~\Cref{ex:septics-needed}, it is natural to ask whether or not rank-2 triples satisfying cubics, quintics, septics but not quartics exist.
This turns out not to be the case: letting $I_1 (F_{ij})$ denote the ideal of $2\times 2$ minors of $F_{ij},$ we find computationally that
\begin{equation}\label{eq:no-full-rank-quartics}
\Big\langle 
\text{cubics, quintics, septics}
\Big\rangle  : \Big(
\big\langle 
\text{quartics}
\big\rangle
\cdot 
I_1(F_{12} ) \cdot 
I_1(F_{13} ) \cdot 
I_1(F_{23} )
\Big)^{\infty}
= \langle 1 \rangle .
% \big\langle 
% \text{eqs~\eqref{eq:quartics}
% }
% \big\rangle \subset \Big(
% \big\langle 
% \text{eqs~\eqref{eq:Fij-det},~\eqref{eq:quintics},~\eqref{eq:septics}
% }
% \big\rangle 
% :
% \big\langle 
% \text{eqs~\eqref{eq:quartics}
% }
% \big\rangle 
% \Big) : \Big(
% I_1(F_{12} ) \cdot 
% I_1(F_{13} ) \cdot 
% I_1(F_{23} )
% \Big)^{\infty}.
\end{equation}
In fact, the same holds if we exchange quartics with quintics above:
\begin{equation}\label{eq:no-full-rank-quintics}
\Big\langle 
\text{cubics, quartics,  septics}
\Big\rangle  : \Big(
\big\langle 
\text{quintics}
\big\rangle
\cdot 
I_1(F_{12} ) \cdot 
I_1(F_{13} ) \cdot 
I_1(F_{23} )
\Big)^{\infty}
= \langle 1 \rangle .
% \big\langle 
% \text{eqs~\eqref{eq:quintics}
% }
% \big\rangle \subset \Big(
% \big\langle 
% \text{eqs~\eqref{eq:Fij-det},~\eqref{eq:quartics},~\eqref{eq:septics}
% }
% \big\rangle 
% :
% \big\langle 
% \text{eqs~\eqref{eq:quintics}
% }
% \big\rangle 
% \Big) : \Big(
% I_1(F_{12} ) \cdot 
% I_1(F_{13} ) \cdot 
% I_1(F_{23} )
% \Big)^{\infty}.
\end{equation}
Here is an analogue of~\Cref{ex:quartics-needed} showing that quintics, just like septics and quartics, are needed to cut out $Y_\calF$ set-theoretically.
\begin{example}\label{ex:quintics-needed}
For the triple 
\begin{equation}\label{eq:quintics-needed}
F_{12}  \sim \begin{bmatrix}
1&0&0\\
0&0&0\\
0&0&0
\end{bmatrix},
\quad 
F_{13}  \sim \begin{bmatrix}
\phantom{-}0&-1&-1\\
-2&\phantom{-}1&\phantom{-}1\\
-1&\phantom{-}0&\phantom{-}0
\end{bmatrix},
\quad 
F_{23}  \sim \begin{bmatrix}
0 & -1 & -1 \\
1&\phantom{-}1&\phantom{-}1\\
2&-2&-2
\end{bmatrix},
\end{equation}
all equations in~\Cref{thm:uncalibrated-ideal} except four quintics vanish.
Note $\rank (F_{12})=1,$ and 
$\rank \left( \left[  F_{31} \, \vert \, F_{32}\right]\right) = 2$.
All such examples follow a similar pattern.
\end{example}

\begin{remark}\label{remark:quartics-from-Martyshev-cubics}
In view of \Cref{thm:calibrated-local}, the quartics \eqref{eq:quartics} must follow from the equations listed in \Cref{remark:result-of-Martyushev}. Indeed, we multiply the transpose of \eqref{eq:necF2} by $\hat F_{jk}$ on the right to get 
\[
 (\hat{F}_{ij}\hat{F}_{jk})^T \hat{F}_{ij}\hat{F}_{jk} - \frac{1}{2}\trace(\hat{F}_{ij}^T \hat{F}_{ij})\, \hat{F}_{jk}^T \hat{F}_{jk} + \hat{F}_{ki}\adj(\hat{F}_{ij})^T \hat{F}_{jk} = 0
\]
and conclude that the last term, $\hat{F}_{ki}\adj(\hat{F}_{ij})^T \hat{F}_{jk}$, must be a symmetric matrix just as in \Cref{remark:quartics-from-symmetric-matrices}.
\end{remark}

\section{Proofs of Main Results}\label{sec:proofs}

To prove~\Cref{thm:uncalibrated-ideal,thm:calibrated-local}, we will need to know the dimensions of $Y_{\calF }$ and $Y_{\calE}.$

\begin{proposition}\label{prop:dimensions}
We have the dimension formulae
\begin{equation}\label{eq:compat-dimensions-E}
\dim (Y_\calE) = 11,
\quad 
\dim (Y_\calF) = 18.
\end{equation}
\end{proposition}

\begin{proof}
As the projective reconstruction theorem states that the generic fibers of the parametrization $\left( \PP (\CC^{3\times 4}) \right)^{\times 3} \dashrightarrow Y_\mathcal{F}$ are single $\PGL_4 (\CC )$-orbits, we have
\[
\dim (Y_\calF ) = \dim \left( \PP (\CC^{3\times 4 } )^{\times 3}  \right) - \dim (\PGL_4(\CC)) = 3\times 11 - 15 = 18.
\]
Similarly, the generic fiber over a point in $Y_\mathcal{E}$ is the union of two orbits under the similarity subgroup of $\PGL_4 (\CC )$, so that
\[
\dim (Y_\calE ) =  3\times 6 - 7 = 11.
\]
\end{proof}

\begin{proof}[Proof of~\Cref{thm:uncalibrated-ideal}]
Let $I$ be the ideal generated by cubics, quartics, quintics,
and septics.
Noting one inclusion $I\subset \mathcal{I}(Y_\calF)$, \Cref{prop:dimensions} implies that the reverse inclusion will follow if we show that $I$ is a prime ideal of Krull dimension $18+3=21.$
Macaulay2 gives us $\dim (I)=21$ with ease, but showing $I$ is prime takes more work.
Consider a primary decomposition,
\begin{equation}\label{eq:p-decomp}
I = Q_0 \cap Q_1 \cap \ldots Q_{s} \cap Q_{s+1} \cap \cdots \cap Q_{s+t},
\end{equation}
where $\dim (Q_0)=\ldots = \dim (Q_s) =21$, $\sqrt{Q_0}=\mathcal{I}(Y_\calF)$, and $\dim (Q_i)<21$ for all $i>s.$
Let $m$ denote the multiplicity of $Q_0$ along $\mathcal{I}(Y_\calF).$
For any $\boldsymbol{\alpha}\in  \ZZ_{\ge 0}^{\times 3},$ we have 
\begin{align*}
\mdeg_I (\boldsymbol{\alpha}) &=
\mdeg_{Q_0} (\boldsymbol{\alpha}) + \mdeg_{Q_1}(\boldsymbol{\alpha}) + \ldots + \mdeg_{Q_s}(\boldsymbol{\alpha}) \\
&= m \cdot \mdeg_{\calI (Y_\calF)} (\boldsymbol{\alpha}) + \mdeg_{Q_1}(\boldsymbol{\alpha}) + \ldots + \mdeg_{Q_s}(\boldsymbol{\alpha})\\
&\ge \mdeg_{\calI (Y_\calF)} (\boldsymbol{\alpha}).
\end{align*}
For this last quantity to be nonzero, $\boldsymbol{\alpha}$ must be one of the lattice points described in the theorem, i.e.~$\boldsymbol{\alpha}$ must be $(7,7,4), (7,6,5), (6,6,6),$ or some permutation thereof.
Macaulay2 readily verifies that the nonzero numbers $\mdeg_I (\boldsymbol{\alpha})$ are exactly those shown in~\Cref{fig:multidegrees} (left.)
We can also prove that these numbers are \emph{lower bounds} on the corresponding numbers $\deg_{\calI (Y_\calF)} (\boldsymbol{\alpha})$ via symbolic computations involving zero-dimensional ideals.\footnote{We refer to the supplementary code linked above for details.}
Thus, we obtain the multidegree of $Y_\calF$,  and moreover $m=1,$  $Q_1=Q_2 \ldots = Q_s=\emptyset$, and $Q_0 = \mathcal{I} (Y_\calF ).$

Finally, we must argue that the lower-dimensional components are all trivial, 
\[
Q_{s+1} = \ldots = Q_{s+t} = \emptyset .
\] 
This can be established using the homological criteria of~\cite{EHV}: if\begin{equation}\label{eq:ext-number}
6=27-21 = \codim (I) = 
\displaystyle\min \left\{ d \mid 
\codim \Ext^d (R/I, R) = d 
\right\}, 
\end{equation}
then $R/I$ is Cohen-Macaulay and hence equidimensional.
The Macaulay2 command \texttt{removeLowestDimension} quickly establishes~\eqref{eq:ext-number}, thus proving $I=Q_0=\calI (Y_\calF)$.
\end{proof}

\begin{proof}[Proof of~\Cref{thm:calibrated-local}]
By~\Cref{prop:dimensions},we may evaluate 
the Jacobian $J(F)$ of equations~\eqref{eq:demazure},~\eqref{eq:quartics}, and~\eqref{eq:martyushev-sextic} at $F_0\in \widehat{Y}_\calE$, and verify $\rank \left(J(F_0) \right)= 11 +3 = 14.$
\end{proof}

Using numerical monodromy heuristics (as described in~\cite{duff-monodromy}), we have computed the multidegree of $Y_{\calE}$ shown in \Cref{fig:multidegrees} (right).
In principal, the output of these calculations may be certified to prove lower bounds on the multidegrees.
To obtain the corresponding upper bounds, a better understanding of the vanishing ideal $\calI (Y_{\calE})$ would needed.
However, we can at least prove that the numerical results are correct for the smallest multidegree.

\begin{proposition}\label{prop:essential-multidegree}
$\mdeg_{Y_\calE}(5,5,1)= \mdeg_{Y_\calE}(5,1,5)= \mdeg_{Y_\calE}(1,5,5)=400.$
\end{proposition}
\begin{proof}
First, by $\simgrp $-invariance, note that $Y_\calE$ is also the image of the map
\begin{align*}
\widetilde{\Psi_\calE} : 
\SO_3 (\CC)^{\times 2} \times \left(\CC^{3}\right)^{\times 2} &\dashrightarrow 
\PP (\CC^{3\times 3})^{\times 3}
\\
(R_{12}, R_{13}, c_{12}, c_{13}) &\mapsto 
\left(
[c_{12}]_{\times } R_{12}^T,
[c_{13}]_{\times } R_{13}^T
,
R_{12} [c_{13}-c_{12}]_{\times } R_{13}^T
\right). \nonumber 
\end{align*}
The corresponding map for a single essential matrix,
\begin{align*}
\Phi_{\calE } : \SO_3 (\CC)\times \CC^{3} &\dashrightarrow 
\PP (\CC^{3\times 3})
\\
(R, c) &\mapsto 
[c]_{\times } R^T
, \nonumber 
\end{align*}
is well-understood: its closed image is a projective variety of dimension $5$ and degree $10,$ and a general fiber of $\Phi_{\calE}$ is the union of two lines,
\[
\Phi_{\calE}^{-1} \left( 
\Phi_{\calE} (R,c)
\right) = 
\{  (R, uc) \mid u \in  \CC^\ast \}
\displaystyle\cup 
\{ (\tw(R, c), uc) \mid u \in  \CC^\ast \}
,
\]
where $\tw$ denotes the ``twisted pair map" (cf.~\cite[eq.~(2)]{galois-siaga}),
\begin{align}\label{eq:twisted-pair}
\tw : \SO_3 (\CC)\times \CC^{3} &\dashrightarrow   \SO_3 (\CC)\times \CC^{3} \\
(R,c) &\mapsto \left( R \left( 
\displaystyle\frac{2}{c^Tc} \, c c^T - I 
\right), -c\right).
\end{align}
Thus, intersecting $Y_{\calE}$ with $5$ general linear spaces in each of the first two projective factors gives $10 \times 10 = 100$ values for the pair $(F_{12}, F_{13})$, and $F_{23}$ must lie on one of $400$ lines $\ell$, which for fixed $(R_{1i},c_{1i}) \in \Phi_\calE^{-1} (F_{1i})$, $i=2,3,$ have the form
\[
\ell = \{ 
F_{23} \in \PP (\CC^{3\times 3}) \mid F_{23} \sim 
R_{12} [ac_{13}-bc_{12}]_{\times } R_{13}^T,
\quad 
[a:b] \in \PP^1 \}.
\]
Intersecting such an $\ell $ with a general hyperplane, we deduce $\mdeg_{Y_{\calE}} (5,5,1) = 400,$ and the other values follow by symmetry.
\end{proof}

\section{Open Problems}\label{sec:open-problems}

We conclude by stating some open problems.

\begin{problem}
Determine polynomial equations that set-theoretically cut out the viewing graph variety $\overline{\Psi_G}$ of~\eqref{eq:parametrize-viewing-graph}, for any viewing graph $G.$
\end{problem}

Beyond the case of $G=K_3$, for which~\Cref{thm:uncalibrated-ideal} provides a stronger ideal-theoretic answer, independent $K_4$ polynomials have been studied in~\cite{DBLP:conf/iccv/BratelundR23,CONNELLY2025102446}.
One may hope that combining these with the polynomials~\Cref{thm:uncalibrated-ideal} solves the problem.

Moving on from the uncalibrated the calibrated case,~\Cref{thm:calibrated-local} represents our current state of knowledge when $G=K_3.$
One may hope to improve this result.

\begin{problem}
    Find set-theoretic equations for the variety $Y_\calE .$
\end{problem}

For the equi-uncalibrated case, we lack even a weak result like~\Cref{thm:calibrated-local}.

\begin{problem}
Determine equations that locally cut out the variety $Y_{\Kdel}.$ 
\end{problem}

Finally, we point out that compatibility varieties are just one of the known approaches to studying the algebraic dependencies between three or more cameras.
An alternative, well-studied approach is via multiple-view tensors~\cite[Ch.~ 17]{DBLP:journals/ijcv/HartleyK07}, also known as \emph{Grassmann tensors}~\cite[\S 4.3.6]{kileel2022snapshot}.
In the Grassmann tensor formalism, the fundamental matrix is known as the bifocal tensor.
Two other species of Grassmann tensor are relevant for the pinhole camera: the \emph{trifocal} and \emph{quadrifocal tensor}.

For uncalibrated cameras, the vanishing ideals of trifocal tensors have been completely determined by the work of Aholt and Oeding~\cite{aholt-oeding}; see also~\cite{DBLP:conf/nips/MiaoL024} for more recent developments connected to tensor decomposition.
Much less is known for the \emph{calibrated trifocal variety}\cite{kileel-siaga}. 
Although the work of Martyushev provides a characterization over the real numbers~\cite{DBLP:journals/jmiv/Martyushev17}, and some low-degree equations, a set-theoretic description of this variety still seems to be unknown.

For quadrifocal tensors, even less is known---see~\cite{oeding-quadrifocal} for a taxonomy of low-degree equations in the uncalibrated case.

\begin{problem}
    Study the relationship between compatibility varieties for the viewing graphs $K_3$ and $K_4$ and the corresponding Grassmann tensor varieties.
    Can we transfer knowledge of  equations from one class of varieties to the other?
\end{problem}

\section*{Acknowledgments}
We are thankful to the Fields Institute for hosting us at the Workshop on the Applications of Commutative Algebra, which helped to initiate this project. The research of AL is partially supported by NSF DMS award 2001267 and the Simons Fellows program. TP was supported by OPJAK CZ.02.01.01/00/22 008/0004590 Roboprox Project. VK was supported by CEDMO 2.0 NPO Project.
\providecommand{\bysame}{\leavevmode\hbox to3em{\hrulefill}\thinspace}
\providecommand{\MR}{\relax\ifhmode\unskip\space\fi MR }
\providecommand{\MRhref}[2]{%
  \href{http://www.ams.org/mathscinet-getitem?mr=#1}{#2}
}
\providecommand{\href}[2]{#2}

\end{document}